\documentclass[12pt, reqno]{amsart}

\usepackage{amsmath}
\usepackage{amsfonts}
\usepackage{a4wide}
\usepackage{amssymb}
\usepackage{amsthm}
\usepackage[colorlinks, citecolor=blue, linkcolor=red]{hyperref}

\theoremstyle{plain}

\newtheorem{teo}{Theorem}[section]
\newtheorem{lemma}[teo]{Lemma}

\newtheorem{cor}[teo]{Corollary}
\newtheorem{ackn}{Acknowledgments\!}

\theoremstyle{definition}

\theoremstyle{remark}

\newtheorem{rem}[teo]{Remark}

\def\Ric{{\mathrm {Ric}}}
\def\ricc{{\mathrm {Ric}}}
\def\dRic{\overset{\circ}{\mathrm{Ric}}}

\def\a{\alpha}

\numberwithin{equation}{section}

\def\SS{{{\mathbb S}}}

\def\RR{{\mathbb R}}

\def\Ric{{\mathrm {Ric}}}

\def\a{{\alpha}}

\def\eps{\varepsilon}

\newcommand{\set}[1]{{\left\{#1\right\}}}               
\newcommand{\pa}[1]{{\left(#1\right)}}                  
\newcommand{\abs}[1]{{\left|#1\right|}}                 
\newcommand{\pair}[1]{\left\langle#1\right\rangle}      

\title[Rigidity of Critical Metrics for Quadratic Curvature Functionals]{Rigidity of Critical Metrics for \\ Quadratic Curvature Functionals}
\date{\today}

\author[Giovanni Catino]{Giovanni Catino}
\address[Giovanni Catino]{Dipartimento di Matematica, Politecnico di Milano, Piazza Leonardo da Vinci 32, 20133 Milano, Italy}
\email[]{giovanni.catino@polimi.it}

\author[P. Mastrolia]{Paolo Mastrolia}
\address[Paolo Mastrolia]{Dipartimento di Matematica, Universit\`{a} degli Studi di Milano, Via Saldini 50, 20133 Italy.}
\email[]{paolo.mastrolia@unimi.it}

\author[D. D. Monticelli]{Dario D. Monticelli}
\address[Dario Monticelli]{Dipartimento di Matematica, Politecnico di Milano, Piazza Leonardo da Vinci 32, 20133 Milano, Italy}
\email[]{dario.monticelli@polimi.it}

\begin{document}

\begin{abstract}
In this paper we prove new rigidity results for complete, possibly non-compact, critical metrics of the quadratic curvature functionals $\mathfrak{F}^{2}_t = \int |\operatorname{Ric}_g|^{2} dV_g +  t \int R^{2}_g  dV_g$, $t\in\mathbb{R}$, and  $\mathfrak{S}^2 = \int R_g^{2} dV_g$. We show that (i) flat surfaces are the only critical points of $\mathfrak{S}^2$, (ii) flat three-dimensional manifolds are the only critical points of $\mathfrak{F}^{2}_t$ for every $t>-\frac{1}{3}$, (iii) three-dimensional scalar flat manifolds are  the only critical points of $\mathfrak{S}^2$ with finite energy  and (iv) $n$-dimensional, $n>4$, scalar flat manifolds are the only critical points of $\mathfrak{S}^2$ with finite energy and scalar curvature bounded below. In case (i), our proof relies on rigidity results for conformal vector fields and an ODE argument; in case (ii) we draw upon some ideas of M. T. Anderson concerning regularity, convergence and rigidity of critical metrics; in cases (iii) and (iv) the proofs are self-contained and depend on new pointwise and integral estimates.
\end{abstract}

\maketitle

\begin{center}


\medskip

\centerline{\bf AMS subject classification: 53C21, 53C24, 53C25}

\end{center}

\


\section{Introduction}
It is a natural problem in Riemannian geometry to study canonical metrics arising as solutions of the Euler-Lagrange equations for curvature functionals.
In~\cite{berger3}, Berger commenced the study of Riemannian functionals which are quadratic in the curvature (see~\cite[Chapter 4]{besse} and~\cite{smot}, \cite{CMbook} for surveys). To fix the notation, let $M^{n}$, $n\geq 2$, be a $n$--dimensional smooth  manifold without boundary. Given a Riemannian metric $g$ on $M^n$, we denote with $\operatorname{Riem_g}$, $W_g$, $\ricc_g$ and $R_g$, respectively, the Riemann curvature tensor, the Weyl tensor, the Ricci tensor and the scalar curvature. Then a basis for the space of quadratic curvature functionals, defined on the space of smooth metrics on $M^n$, is given by
$$
\mathfrak{W}^2 = \int |W_g|^{2} dV_g\,, \qquad \mathfrak{r}^2 = \int |\ricc_g|^{2} dV_g\,, \qquad\, \mathfrak{S}^2 = \int R_g^{2} dV_g.
$$
It is well known that the only quadratic functional in the case $n=2$ is given by $\mathfrak{S}^2$, while in dimension $n=3$ one only has $\mathfrak{S}^2$ and $\mathfrak{r}^2$.
From the standard decomposition of the Riemann tensor, for every $n\geq4$, one has
$$
\mathfrak{R}^2 = \int |\operatorname{Riem_g}|^{2} dV_g = \int \left(|W_g|^{2}+\frac{4}{n-2}|\ricc_g|^{2}-\frac{2}{(n-1)(n-2)}R_g^{2}\right) dV_g \,.
$$
All such functionals, which also arise naturally as total actions in certain gravitational fields theories in physics, have been deeply studied in the last years by many authors; they are also of great interest in the K\"ahler framework: in this case, restricting to metrics in a given cohomology class, it is possible to show that each one of the three basis functionals can be expressed affinely by any one of them (see the discussion in \cite[Chapter 11.E.]{besse}).

In this paper we focus our attention on rigidity results for critical metrics of quadratic curvature functionals which do not depend on $\mathfrak{W}^2$, i.e. we will consider
$$
\mathfrak{F}^{2}_t = \int |\ricc_g|^{2} dV_g +  t \int R^{2}_g  dV_g \,,
$$
defined for some $t\in \RR\cup\{+\infty\}$, the case  $t = +\infty$ formally corresponding to the functional $\mathfrak{S}^2$.

Using variations with compact support, one finds that the Euler--Lagrange equation for a critical metric of $\mathfrak{F}^{2}_t$ reads, in local coordinates, as
\begin{equation}\label{eq-var-ft}
-\Delta R_{ij} +(1+2t)\nabla^{2}_{ij} R_g - \frac{1+4t}{2}(\Delta R_g)g_{ij} + \frac{1}{2} \Big( |\ricc_g|^{2}+ t R_g^{2}\Big) g_{ij} - 2 R_{ikjl}R_{kl} -2t R_g R_{ij}=0\,,
\end{equation}
which upon tracing yields
\begin{equation}\label{eq-var-ft-tr}
 \Big(n+4(n-1)t\Big)\Delta R_g \,=\,(n-4)\Big(|\ricc_g|^{2}+ t R_g^{2}\Big) \,.
\end{equation}
In particular for $\mathfrak{S}^2$ we obtain
\begin{equation}\label{eqbis}
R_g \,\ricc_g - \nabla^{2} R_g \,=\, \frac{3}{4(n-1)} R_g^{2}\,g \,,
\end{equation}
\begin{equation}\label{eq1bis}
\Delta R_g \,=\, \frac{n-4}{4(n-1)} R_g^{2} \,,
\end{equation}
It is clear that for $n=2$ and $n=3$ one expects stronger results, since the systems of equations \eqref{eq-var-ft} and \eqref{eqbis} govern the full curvature of the manifold. On the other hand, in dimensions $n\geq4$ one would surmise that information on the full Riemann tensor might be needed, at least to study system \eqref{eq-var-ft}.

\

Since in dimensions different from four $\mathfrak{F}^{2}_t$ is not scale-invariant, when $M^n$ is compact it is natural to restrict the functional on $\mathcal{M}_1(M^{n})$, the space of equivalence classes of Riemannian metrics on $M^n$ having unit volume. Equivalently, one can consider a modified functional properly normalized with the volume of the manifold. It was already observed in~\cite{besse} that every Einstein metric is critical for $\mathfrak{F}^{2}_t$ on $\mathcal{M}_1(M^{n})$, for every $t\in\RR$. The converse in general is false, of course. For instance, in dimension four, every Bach-flat metric is critical for $\mathfrak{F}^2_{-1/3}$ and every Weyl and scalar flat metric is critical for $\mathfrak{F}^2_{-1/4}$ on $\mathcal{M}_1(M^{4})$ (see~\cite[Chapter 4]{besse}). Moreover, Lamontagne in ~\cite{lamontagne1} constructed a homogeneous non-Einstein critical metric for $\mathfrak{R}=4\mathfrak{F}^2_{-1/4}$ on $\mathcal{M}_1(\SS^{3})$. This construction can be generalized to the case when $t>-\frac{1}{2}$, see \cite[section 7]{gurviastab}. Other examples of compact critical metrics were constructed by Gursky and Viaclovsky in \cite{gurviacon}, where they considered a manifold $M^4$ which is the connected sum of some Einstein manifolds and constructed critical metrics for $\mathfrak{F}^2_t$ on $\mathcal{M}^1(M^4)$, for $t$ ``close'' to a given value depending on the topology of the Einstein building blocks. It is then natural to ask under which conditions a critical metric for $\mathfrak{F}^{2}_t$ must be Einstein. Typically, one assumes some curvature conditions (of pointwise or integral type, positivity or negativity of the curvature, etc.) on the critical metric, in order to prove rigidity properties; for results in this direction see e.g. \cite{and1, gurvia3, cat1, lamontagne2, CatinoQuadratic, garciahomo}.


\

The problem of studying critical metrics for quadratic curvature functionals in the non-compact setting has on the other hand received much less attention. It is easy to see that all Ricci-flat metrics are critical points of the functional $\mathfrak{F}^2_t$ when $t\in\RR$, which are also global minima of the functional when $ t\in[-\frac{1}{n},+\infty)$; similarly, scalar-flat metrics are global minima for $\mathfrak{S}^2$, and hence are critical.

Anderson in~\cite{and2} proved that every complete three-dimensional critical metric for the functional $\mathfrak{F}^2_t$ with non-negative scalar curvature is flat, if $t=0$ or $t=-\frac{1}{3}$. As an auxiliary result, in Lemma \ref{l-rig} we observe that the same techniques allow to extend the result to include the case $t\in(-\frac{1}{3},+\infty)$. In~\cite{cat1} the first author showed a characterization of complete critical metrics for $\mathfrak{S}^2$ with non-negative scalar curvature, in any dimension $n\geq 3$. On the other hand, in~\cite{catmasmonvar}, we showed that, in dimension three, flat metrics are the only complete metrics with non-negative scalar curvature which are critical for the $\sigma_{2}$-curvature functional $\mathfrak{F}^2_{-3/8}$. All these results rely heavily on the sign condition on the scalar curvature. We observe that, to the best of our knowledge, the only example of a smooth critical metric for a functional $\mathfrak{F}^2_{t}$ was constructed by Gursky and Viaclovsky in \cite{gurvia3} when $n=3$ and $t=-\frac{3}{8}$; this example has strictly negative scalar curvature, and we note for later reference that it does not belong to any $L^q$ space.

\

We now proceed to describe our results. As already observed, when $n=2$ the only quadratic curvature functional is $\mathfrak{S}^2$; we then show that the only critical metrics for such functional are flat, without any further assumption. Indeed we have
\begin{teo}\label{t-2d} Let $(M^{2},g)$ be a two-dimensional complete critical metric of $\mathfrak{S}^{2}$. Then $(M^{2},g)$ is flat, and thus a global minimum of the functional.
\end{teo}
The proof of this theorem relies on known results on gradient conformal solitons and on an ODE argument.

When $n=3$, we also prove the same result for critical metrics of the functional $\mathfrak{F}^{2}_t$, with $t\in(-\frac{1}{3},+\infty)$; hence we have
\begin{teo}\label{t-3dFt}
Let $(M^{3},g)$, be a three-dimensional complete critical metric of $\mathfrak{F}^{2}_t$ with
$$-\frac13 < t <+\infty.$$
Then $(M^{3},g)$ is flat, and thus a global minimum of the functional.
\end{teo}
For other values of $t$, in general this result may not hold, as the explicit example constructed by Gursky and Viaclovsky in \cite{gurvia3} shows for $t=-\frac{3}{8}$. The proof of the theorem follows essentially using arguments introduced by Anderson in \cite{and1, and2}. Notice that the result in Theorem \ref{t-3dFt} holds also for $t=-\frac{1}{3}$ under the additional assumption that the scalar curvature is non-negative, as it was shown by Anderson in \cite{and2}.

We explicitly note that the rigidity results in Theorems \ref{t-2d} and \ref{t-3dFt} hold without any extra assumption on the manifold and its curvature.

For the case of the functional $\mathfrak{S}^{2}$ when $n=3$, which formally corresponds to $t=+\infty$, we need an integrability assumption on $R_g$. Indeed Anderson's techniques, which we used in the proof of Theorem \ref{t-3dFt}, cannot be directly applied to this functional, as detailed in Remark \ref{rem-not-And}.
\begin{teo}\label{t-3d}
Let $(M^{3},g)$, be a three-dimensional complete critical metric of $\mathfrak{S}^{2}$ with finite energy, i.e. $R_g\in L^2(M^3)$. Then $(M^{3},g)$ is scalar flat, and thus a global minimum of the functional.
\end{teo}
\begin{rem}\label{r-3d} We actually prove a stronger result: let $(M^{3},g)$ be a three-dimensional complete, non-compact, critical metric of $\mathfrak{S}^{2}$ with $R_g\in L^q(M^3)$ for some $q\in
(1,\infty)$. Then $(M^{3},g)$ is scalar flat, and thus a global minimum of the functional.
\end{rem}
We explicitly note that here we do not assume any sign condition on the scalar curvature. In the compact case the result easily follows integrating the trace of the Euler-Lagrange equation of the functional. In the non-compact case the proof consists in two steps: first, using a test function argument, we show that under our hypotheses one must have that $R_g\geq0$, then we conclude using a previous result for manifolds with non-negative scalar curvature by the first author, see \cite{cat1}.

\

In dimension $n=4$ the functionals $\mathfrak{F}^2_t$ and $\mathfrak{S}^2$ are scaling-invariant. In particular for $t=-\frac{1}{3}$ a metric is critical if and only if it is Bach-flat, and hence there are several well known examples of solutions in the literature, both in the compact and in the complete non-compact case, e.g. conformally Einstein metrics and (anti-)self-dual metrics. In particular, in \cite{tiavia} the authors study complete, non-compact Bach-flat metrics with finite energy. In case $t\neq-\frac{1}{3}$ or $t=+\infty$, from equations \eqref{eq-var-ft-tr} and \eqref{eq1bis} one immediately sees that the metric has harmonic scalar curvature. In particular, in the compact case one has that $R_g$ must be constant. We note that examples of compact critical metrics of $\mathfrak{F}_t^2$ for some values of $t\neq-\frac{1}{3}$ which are not Einstein have been constructed in \cite{gurviacon}. On the other hand, one can deduce from \eqref{eqbis} that compact critical metrics of $\mathfrak{S}^2$ are Einstein. In the non-compact case a classical result of Yau~\cite{yau3} implies that a critical metric with $R_g\in L^{q}(M^{4})$ for some $q\in(1,+\infty)$ must have constant scalar curvature; then $(M^4,g)$ is scalar flat or it has finite volume. Moreover from equation~\eqref{eqbis} one has that a critical metric for $\mathfrak{S}^2$ is either scalar flat or Einstein with negative curvature and finite volume.

\

Finally, in the case $n>4$, it is clear that one cannot hope to prove rigidity results for critical metrics of $\mathfrak{F}_t^2$ without assuming further conditions on the full curvature tensor of the metric. Therefore, we focus our attention on $\mathfrak{S}^2$ and we prove that all critical metrics are scalar flat, assuming that the scalar curvature is bounded below and that it satisfies an integrability condition.
\begin{teo}\label{t-hd}
Let $(M^{n},g)$, $n>4$, be a complete critical metric of $\mathfrak{S}^{2}$ with finite energy, i.e. $R_g\in L^2(M^n)$, and with $R_g$ bounded from below on $M^n$. Then $(M^{n},g)$ is scalar flat, and thus a global minimum of the functional.
\end{teo}
\begin{rem}\label{r-hd} Actually we prove a stronger result: there exists $q^*>2$ such that if $(M^{n},g)$, $n>4$, is a complete critical metric of $\mathfrak{S}^{2}$ with $R_g$ bounded from below on $M^n$ and with $R_g\in L^q(M^n)$ for some $1<q<q^*$, then $(M^{n},g)$ is scalar flat. See the proof of Theorem \ref{t-hd} and Remark \ref{remf}.
\end{rem}
We explicitly note that also in this case we do not assume any sign condition on $R_g$. The proof of the result does not rely on previous works on critical metrics for quadratic curvature functionals and is divided into three steps: first, using a test function argument, we show that under our assumptions $R_g$ must either vanish identically or be strictly negative on $M^n$; then we prove a gradient estimate for the function $R_g$, under the assumption that it is strictly negative on $M^n$, using the Omori-Yau maximum principle for the $f$-Laplace operator and the $f$-Laplacian comparison theorem for the function $f=-\log|R_g|$; the third and final step, which relies on the previous gradient estimate, consists in showing that the case $R_g<0$ on $M^n$ cannot occur, using another test function argument.

\

The rest of the paper is organized as follows: in Section \ref{sec2} we derive the Euler-Lagrange equations associated to $\mathfrak{F}^2_t$ and to $\mathfrak{S}^2$, in any dimension; in Section \ref{sec3} we consider the case $n=2$ and we prove Theorem \ref{t-2d}; in Section \ref{sec4} we study the case $n=3$ and we provide the proofs of Theorems \ref{t-3dFt} and \ref{t-3d} (including Remark \ref{r-3d}); finally in Section \ref{sec5} we consider the case when $n>4$ and we show Theorem \ref{t-hd} (including Remark \ref{r-hd}).

\

\

\begin{ackn}
\noindent The authors would like to thank Prof. M. T. Anderson for many helpful discussions.

The first and second authors are members of the {\em GNSAGA, Gruppo Nazionale per le Strutture Algebriche, Geometriche e le loro Applicazioni} of Indam. The third author is a member of {\em GNAMPA, Gruppo Nazionale per l'Analisi Matematica, la Probabilità e le loro Applicazioni} of Indam and has been partially supported by {\em 2020 GNAMPA Project: Equazioni Ellittiche e Paraboliche in Analisi Geometrica}.
\end{ackn}

\

\section{Euler-Lagrange Equations}\label{sec2}

From now on we will drop the subscript $g$ in the notation of geometric objects. In this section we will follow closely the presentation in \cite{CatinoQuadratic} and \cite{CMbook}.

We fix the index range
$1\leq i, j, \ldots \leq n$ and recall that the Einstein summation
convention will be in force throughout.
The standard decomposition of the $(0,
4)$-version of the Riemann tensor   is given by the formula
\begin{equation}\label{Riemann_Weyl}
  R_{ijkt} = W_{ijkt} + \frac{1}{n-2}\pa{R_{ik}g_{jt}-R_{it}g_{jk}+R_{jt}g_{ik}-R_{jk}g_{it}}-\frac{R}{(n-1)(n-2)}\pa{g_{ik}g_{jt}-g_{it}g_{jk}},
\end{equation}
where $g_{ij}$, $W_{ijkt}$ and $R_{ij}$ are the components, respectively, of the metric $g$, of the Weyl tensor $\operatorname{W}$ and of the Ricci tensor $\ricc$, while $R$ is the scalar curvature. Note that
 $R_{ij} = g^{kl}R_{ikjl} = R_{kilj}$ and $R = g^{ij}R_{ij}$, where $g^{ij}$  are the components of the inverse of metric $g$, $g^{-1}$.

%
%

We  consider the quadratic curvature functionals
$$
\mathfrak{F}^{2}_t = \int |\ricc|^{2} dV +  t \int R^{2}  dV \,,
$$
defined for some constant $t\in \RR$ (with $t = +\infty$ formally corresponding to the functional $\mathfrak{S}$).
Using formulas for the variations of the Ricci tensor and of the scalar curvature  and integrating by parts (using compactly supported ``directions'' $h$), it is not difficult to show that the so-called gradients of the functionals $\mathfrak{r}^2$ and $\mathfrak{S}^2$ are given by (see also~\cite[Proposition 4.66]{besse})
$$
(\nabla \mathfrak{r}^2)_{ij} = -\Delta R_{ij} - 2 R_{ikjl}R_{kl}+\nabla^{2}_{ij} R - \frac{1}{2}(\Delta R)g_{ij} + \frac{1}{2} |\ricc|^{2} g_{ij} \,,
$$

$$
(\nabla \mathfrak{S}^2)_{ij} = 2 \nabla^{2}_{ij} R - 2(\Delta R) g_{ij} -2 R R_{ij} + \frac{1}{2} R^{2} g_{ij} \,;
$$
hence, the gradient of $\mathfrak{F}^{2}_t$ reads
\begin{equation}
(\nabla \mathfrak{F}^{2}_t)_{ij} = -\Delta R_{ij} +(1+2t)\nabla^{2}_{ij} R - \frac{1+4t}{2}(\Delta R)g_{ij} + \frac{1}{2} \Big( |\ricc|^{2}+ t R^{2}\Big) g_{ij} - 2 R_{ikjl}R_{kl} -2t R R_{ij} \,,
\end{equation}
and consequently the Euler--Lagrange equation for a critical metric of $\mathfrak{F}^{2}_t$ are given by

\begin{equation}\label{eq-t}
-\Delta R_{ij} +(1+2t)\nabla^{2}_{ij} R - \frac{1+4t}{2}(\Delta R)g_{ij} + \frac{1}{2} \Big( |\ricc|^{2}+ t R^{2}\Big) g_{ij} - 2 R_{ikjl}R_{kl} -2t R R_{ij}=0.
\end{equation}
Tracing the equation $(\nabla \mathfrak{F}^{2}_t) = 0$, we obtain
\begin{equation}\label{eq-tt}
 \Big(n+4(n-1)t\Big)\Delta R \,=\,(n-4)\Big(|\ricc|^{2}+ t R^{2}\Big) \,.
\end{equation}

Moreover, the Euler--Lagrange equation for a critical metric of $\mathfrak{S}^2$ is given by
$$
2R \, \ricc - 2\nabla^{2} R + 2\Delta R\, g \,=\, \frac{1}{2} R^{2}\,g \,,
$$
or, equivalently,
\begin{equation}\label{eq}
R \,\ricc - \nabla^{2} R \,=\, \frac{3}{4(n-1)} R^{2}\,g \,,
\end{equation}
\begin{equation}\label{eq1}
\Delta R \,=\, \frac{n-4}{4(n-1)} R^{2} \,,
\end{equation}
where equation~\eqref{eq1} is just the trace of~\eqref{eq}.


\

\section{Critical surfaces}\label{sec3}

Let $(M^2, g)$ be a two-dimensional complete critical metric of $\mathfrak{S}^{2}$. Since $\ricc=\frac12 R\,g$, equation \eqref{eq} reads
\begin{equation}\label{eq2d}
\nabla^2 R = -\frac14 R^2\,g.
\end{equation}
In particular $(M^2, g)$ is a special gradient conformal soliton with potential function $R$ and hence $\nabla R$ is a special conformal vector field. Complete Riemannian manifolds admitting a vector field $X$ satisfying
$$
\mathcal{L}_X g = \lambda g
$$
for some smooth function $\lambda$, were studied by many authors in the late 60's (see, for instance, the discussion in \cite{cheegcold, catmanmazcon}).
\begin{proof}[Proof of Theorem \ref{t-2d}]
First of all, if $M^2$ is compact, then the tracing  equation \eqref{eq2d} we get
$$
\Delta R = -\frac34 R^2,
$$
thus, integrating over $M^2$ we get $R\equiv 0$ on $M^2$.

\smallskip

On the other hand, suppose that $M^2$ is non-compact and $g$ is not flat. By \eqref{eq2d} the potential function $f:=R$ is non-constant. Let $\Sigma$ be a regular level set of the function $f:M^2\to\RR$,
i.e. $|\nabla f|\neq 0$ on $\Sigma$, which exists by Sard's Theorem (and the fact that $f$ is non-constant). Following \cite{catmanmazcon} (see also \cite{cheegcold}), one can easily show that, in a neighborhood $U$ of $\Sigma$ not
containing any critical point of $f$, such potential function only depends
on the signed distance $r$ to the hypersurface $\Sigma$. Moreover, the width of the neighborhood $U$ is uniform with respect to the points of $\Sigma$, namely we can assume $U=\{r_{*}<r<r^{*}\}$, for some maximal $r_{*}\in[-\infty,0)$ and $r^{*}\in(0,+\infty]$. By the scalar invariance of equation~\eqref{eq2d}, we can assume that $f'(0)=1$, possibly changing the function $f$. Hence, in $U$, the metric can be written as
\begin{equation}\label{metric}
g \, = \, dr \otimes dr \,+ \,f'(r)^{2}\, g^{\Sigma}\,,
\end{equation}
where $g^{\Sigma}$ denotes the induced metric on the level set $\Sigma$. Since the hessian of $f$ is given by
$$
\nabla^2 f = f'' dr\otimes dr + (f')^2f''\, g^{\Sigma},
$$
from equation \eqref{eq2d}, we get that the potential function satisfies the following ODE
\begin{equation}\label{101}
f''(r)=-\frac14 (f(r))^2,\quad r\in(r_*, r^*).
\end{equation}
It is then immediate to see that
\begin{equation}\label{102}
(f'(r))^2+\frac{1}{6}(f(r))^3=\operatorname{const},\quad r\in(r_*, r^*)
\end{equation}
is a conserved quantity of the ODE \eqref{101}. One can then carry out a standard phase--plane analysis and a qualitative study of the solutions of \eqref{101}. Indeed one can rewrite \eqref{101} as a system
\begin{equation*}
  \begin{cases}
    x'=y, \\
    y'=-\frac{1}{4}x^2,
  \end{cases}
\end{equation*}
with $f=x$, $f'=y$, and note that by \eqref{102} all solutions must lie on level sets of the function $F(x,y)=y^2+\frac{1}{6}x^3$. Moreover one can estimate the maximal interval of existence of solutions of \eqref{101} by rewriting \eqref{102} as $$\frac{df}{\sqrt{c-\frac{1}{6}f^3}}=\pm dr,$$ for $c\in\mathbb{R}$, and integrating. Then we see that all nontrivial solutions of the ODE can be divided into three families, according to their qualitative behavior.

Solutions in the first family are bounded above and have maximal domain that is a bounded open interval. If $f:(r_*,r^*)\rightarrow\mathbb{R}$ is one of these solutions, then $$\lim_{r\rightarrow r_*}f(r)=\lim_{r\rightarrow r^*}f(r)=-\infty,$$ $f'$ is strictly monotone decreasing with  $$\lim_{r\rightarrow r_*}f'(r)=+\infty,\qquad\lim_{r\rightarrow r^*}f'(r)=-\infty,$$ $f$ has a unique critical point $r_0$ and $f(r_0)=\max f(r)\neq0$.

Solutions in the second family are strictly negative and have maximal domain that is a half--line which is unbounded below, i.e. $r_*=-\infty$. If $f:(-\infty,r^*)\rightarrow\mathbb{R}$ is one of the solutions in this family, then $$\lim_{r\rightarrow-\infty}f(r)=0,\qquad\lim_{r\rightarrow r^*}f(r)=-\infty,$$ $f'$ is strictly monotone decreasing with $$\lim_{r\rightarrow -\infty}f'(r)=0,\qquad\lim_{r\rightarrow r^*}f'(r)=-\infty$$ and $f$ has no critical points.

Finally, solutions of the third family are strictly negative and have maximal domain that is a half--line which is unbounded above, i.e. $r^*=+\infty$. If $f:(r_*,+\infty)\rightarrow\mathbb{R}$ is one of these solutions, then $$\lim_{r\rightarrow r_*}f(r)=-\infty,\qquad\lim_{r\rightarrow+\infty}f(r)=0,$$ $f'$ is strictly monotone decreasing with $$\lim_{r\rightarrow r_*}f'(r)=+\infty,\qquad\lim_{r\rightarrow +\infty}f'(r)=0$$ and $f$ also in this case has no critical points.

Now, if $f$ has no critical points, since $(M^2, g)$ is complete, non-compact, the width of the maximal
neighborhood $U$ should be unbounded in both the negative and the positive direction
of the signed distance $r$ (i.e. we should have $r_* = -\infty$ and $r^* = +\infty$). On the other hand, if $f$ has only one critical point, since $(M^2, g)$ is complete, the width of the neighborhood $U$ should be unbounded in
the positive or in the negative direction of the signed distance (i.e. we should have $r_* = -\infty$ or $r^* = +\infty$, while the other endpoint of the maximal domain should be finite). This is clearly in contradiction with the above qualitative study of the solutions of \eqref{101}. Thus, we conclude that necessarily $f=R\equiv0$ and $(M^2, g)$ is flat.

\end{proof}

\

\section{Dimension three}\label{sec4}

\subsection{Anderson's theory}
On a three-dimensional manifold $M^3$ we consider the quadratic curvature functionals
$$
\mathfrak{F}^{2}_t = \int |\ricc|^{2} dV +  t \int R^{2}  dV \,, \quad t\in\mathbb{R}.
$$
From Section \ref{sec2} (see equation \eqref{eq-t}), in dimension three, using $\operatorname{W}=0$ and the decomposition \eqref{Riemann_Weyl}, we have
\[
R_{ikjl}R_{kl}= \frac{3}{2}R R_{ij} + \pa{\abs{\ricc}^2-\frac{1}{2}R^2}g_{ij}-2R_{ik}R_{kj};
\]
thus, the Euler--Lagrange equation for a (smooth) critical metric of $\mathfrak{F}^{2}_t$ reads
\begin{equation}\label{s1}
-\Delta R_{ij} +(1+2t)\nabla^{2}_{ij} R - \frac{1+4t}{2}(\Delta R)g_{ij} - \frac{3}{2}|\ricc|^{2}g_{ij}+\frac{2+t}{2}R^{2} g_{ij} -(3+2t) R R_{ij} + 4R_{ik}R_{kj} = 0,
\end{equation}
which upon tracing yields
\begin{equation}\label{s2}
 \Big(3+8t\Big)\Delta R =-\Big(|\ricc|^{2}+ t R^{2}\Big) \,.
\end{equation}
We will follow closely the theory developed by Anderson in \cite{and1}. We say that $g$ is a weak $L^{2,2}$ solution to the system \eqref{s1}-\eqref{s2} if, for every two-tensor $h\in L^{2,2}$ with compact support, we have
$$
\int_M \left\langle (\nabla \mathfrak{F}^{2}_t)_{ij},h_{ij}\right\rangle\,dV = \int_M \operatorname{tr}(\nabla \mathfrak{F}^{2}_t)\operatorname{tr}(h)\,dV = 0,
$$
where the $L^{2,2}$ topology on the space of metrics is given by the norm
\[
\|h\|^2_{L^{2,2}} = \int_M\pa{\abs{\nabla^2h}^2+\abs{\nabla h}^2+\abs{h}^2}\,dV
\]
 (see \cite[Section 1]{and1} for details).
Of course here one must perform the appropriate (formal) integration by parts to obtain well defined integrals.

The proof of Theorem \ref{t-3dFt} is by contradiction: assuming a critical metric of $\mathfrak{F}_t^2$ is non-flat, one can construct another non-flat critical metric with bounded (Ricci) curvature. By the Omori-Yau maximum principle the solution has non-negative scalar curvature which in turn implies it must be flat, thus reaching a contradiction. The following lemmas are technical steps which are used for the proof of Theorem \ref{t-3dFt} and are extension of results (regularity, blow-up argument and rigidity) proved by M. Anderson for specific values of the parameter $t$.

\begin{lemma}[{\cite[Theorem 4.1, $t=-1/4$]{and1}}]\label{l-regularity} Let $t\neq -\frac38$, $U$ be a domain in a three-dimensional manifold $M^3$ and suppose $g$ is a $L^{2,2}$ weak solution of the Euler--Lagrange equation $\nabla \mathfrak{F}^{2}_t=0$ on $U$, i.e. $g$ is an $L^{2,2}$ critical metric of $\mathfrak{F}^{2}_t$ on $U$. Then $g$ is smooth.
\end{lemma}
\begin{proof}[Proof (sketch)] Due to the form of the system \eqref{s1}-\eqref{s2}, the proof of this lemma is the same as in \cite[Theorem 4.1]{and1} for the case $t=-1/4$ for every $t\in\RR$ and $t\neq -\frac38$. We note that, if $t=-\frac38$, we cannot gain regularity from the traced equation \eqref{s2} and the proof does not work.
\end{proof}

\begin{lemma}[{\cite[Lemma 2.1, $t=0$]{and2}}]\label{l-blowup} Let $(M^3,g)$ be a complete, non-flat, critical solution of $\mathfrak{F}^{2}_t$ for $t\neq -\frac38$. Then there exists another complete, non-flat, critical solution of $\mathfrak{F}^{2}_t$, $(\bar{M}^3, \bar{g})$, which has uniformly bounded curvature, i.e.
$$
|\ricc_{\bar{g}}|_{\bar{g}} \leq 1\quad\text{on } \bar{M}^3.
$$
\end{lemma}
\begin{proof}[Proof (sketch)] Given the local regularity result of Lemma \ref{l-regularity}, the proof of this lemma is exactly the same as in \cite[Lemma 2.1]{and2} for the case $t=0$.
\end{proof}

\begin{lemma}[{\cite[Theorem 0.1, $t=0$]{and2}}]\label{l-rig} Let $(M^3,g)$ be a complete critical solution of $\mathfrak{F}^{2}_t$ for
$$-\frac13< t < +\infty$$
with non-negative scalar curvature. Then $(M^3,g)$ is flat.
\end{lemma}
\begin{proof}[Proof (sketch)] The proof is (substantially) contained in \cite{and2}. The case $t=0$ is treated in full detail, whereas when $t=-\frac13$ in \cite[Proposition 5.4]{and2} the author highlights the differences between the two cases. An examination of the proof for the case $t=0$ (i.e. \cite[Theorem 0.1]{and2}) shows that all the arguments remain valid also if $-\frac13<t<+\infty$. First of all, from Lemma \ref{l-regularity} and Lemma \ref{l-blowup}, local regularity of weak solutions and the possibility to use a solution with bounded curvature are guaranteed. The second important observation is that the traced equation \eqref{s2} reads
$$
\Big(3+8t\Big)\Delta R =-\Big(|\ricc|^{2}+ t R^{2}\Big) =-\left(|\dRic|^{2}+ \frac{1+3t}{3} R^{2}\right)\leq -\frac{1+3t}{3} R^2,
$$
and therefore, if $t>-\frac13$, all the proofs still work. Here $\dRic=\ricc-\frac13 R\,g$. In particular, as observed also in the proof of \cite[Proposition 5.4]{and2} for the case $t=-\frac13$, the following two instances can be solved:
\begin{itemize}
\item[(i)] The passage from $(2.7)$ to $(2.8)$ in the proof of \cite[Proposition 2.2]{and2}. In the general case, $(2.7)$ and $(2.8)$ should read as
$$
\int \eta^4 \left(|\dRic|^{2}+ \frac{1+3t}{3} R^{2}\right) \leq \mu\int \eta^4 R^2 +\mu^{-1}\int \eta^2(\eta')^2(H^+)^2+\mu^{-1}\int \eta^2(\eta'')^2
$$
and
$$
\int_{B(R)}\left(|\dRic|^{2}+ \frac{1+3t-3\mu}{3} R^{2}\right) \leq c_2 R^{-2}\int_{B(2R)}(H^+)^2+ c_3 R^{-\varepsilon},
$$
respectively. Therefore, if $t>-\frac13$ and $\mu$ is sufficiently small, following the proof we obtain that $(N,g)$ satisfies $$\dRic\equiv R\equiv 0,$$ and therefore must be flat, as required.

\smallskip

\item[(ii)] In the proof of \cite[Lemma 2.9 (ii)]{and2}, in our case, the condition $\Delta R(x_i)\to 0$ implies $|\dRic|(x_i)\to 0$ and also $R(x_i)\to 0$, since $t>-\frac13$. Therefore the proof can be completed following the same steps.
\end{itemize}

\end{proof}
\begin{rem}\label{remmdn} The result is true also if $t=-1/3$ (see \cite[Section 5.2]{and2}). Moreover, a similar result holds if $t=+\infty$, i.e. for critical solution of $\mathfrak{S}^2$, as it was shown in \cite{cat1}. In this case $(M^3,g)$ must be scalar flat.
\end{rem}

\begin{proof}[Proof of Theorem \ref{t-3dFt}] Suppose that $(M^{3},g)$ is a three-dimensional complete, non-flat, critical metric of $\mathfrak{F}^{2}_t$ with $-\frac13 < t <+\infty$. From Lemma \ref{l-blowup} we can construct another complete, non-flat, critical solution of $\mathfrak{F}^{2}_t$, $(\bar{M}^3, \bar{g})$, which has uniformly bounded curvature, i.e.
$$
|\ricc_{\bar{g}}|_{\bar{g}} \leq 1\quad\text{on } \bar{M}^3.
$$
In particular the Ricci curvature of $\bar{g}$ is bounded below. Therefore, we can apply the classical Omori-Yau maximum principle (see for instance \cite{yau2}) to the scalar curvature of ${\bar{g}}$ which is bounded below and satisfies the differential inequality
$$
(3+8t)\Delta_{\bar{g}} R_{\bar{g}} =-\left(|\dRic_{\bar{g}}|_{\bar{g}}^{2}+ \frac{1+3t}{3} R_{\bar{g}}^{2}\right) \leq -\frac{1+3t}{3}R_{\bar{g}}^2,
$$
obtaining
$$
\inf_{\bar{M}} R_{\bar{g}} \geq 0.
$$
Lemma \ref{l-rig} applied to $(\bar{M}^3, \bar{g})$ implies that $(\bar{M}^3, \bar{g})$ must be flat, a contradiction. Therefore the critical solution $(M^{3},g)$ is flat and Theorem \ref{t-3dFt} is proved.
\end{proof}

\

\subsection{Critical metrics of $\mathfrak{S}^2$}
The proofs of Theorem \ref{t-3d} and Remark \ref{r-3d} rely upon an integral estimate for critical metrics of $\mathfrak{S}^2$, that we now prove in general dimension $n$.
We recall that a complete Riemannian manifold $(M^{n},g)$ is critical for $\mathfrak{S}^2$ if it  satisfies ~\eqref{eq} and \eqref{eq1}, i. e.
$$
R \,\ricc - \nabla^{2} R \,=\, \frac{3}{4(n-1)} R^{2}\,g \,.
$$
and
$$
\Delta R \, = \, \frac{n-4}{4(n-1)} R^{2} \,.
$$

We have the following estimate:

\begin{lemma}\label{estimate} Let $(M^{n},g)$ be a complete critical metric of $\mathfrak{S}^2$. Assume  that there exists a point $O\in M^{n}$ such that $R(O)<0$ and,
for $s>0$,  define the open set
$$
M^{-}_{s} = \{ p\in M^{n}|\, R(p)<0 \} \cap B_{s} (O) \,,
$$
where $B_{s}(O)$ is the geodesic ball of radius $s$ centered in $O$. Then, for every $0<s_{1}<s_{2}$ and every $\a>-1$, the following estimate holds:
\begin{eqnarray*}
\int_{M^{-}_{s_{1}}}|\nabla R|^{2} |R|^{\a} \,dV_{g}\leq \frac{(n-4)}{2(n-1)(1+\a)} \int_{M^{-}_{s_{2}}}|R|^{\a+3}\,dV_{g} + \frac{4}{(1+\a)^{2}(s_{2}-s_{1})^{2}}\int_{M^{-}_{s_{2}}}|R|^{\a+2} \,dV_{g} \,.
\end{eqnarray*}

\end{lemma}

\begin{proof} Let $\eta$ be a smooth cutoff function such that $\eta\equiv1$ on $B_{s_1}(O)$, $\eta\equiv0$ on $B_{s_2}^c(O)$, $0\leq\eta\leq1$ on $M^n$ and $|\nabla\eta|\leq\frac{c}{s_2-s_1}$, $c>0$ independent of $s_1,s_2$. Integrating by parts one obtains
\begin{eqnarray*}
\int_{M^{-}_{s_{2}}}|\nabla R|^{2} |R|^{\a} \eta^{2}\,dV_{g} &=& \int_{M^{-}_{s_{2}}} \langle \nabla R, \nabla R \rangle (-R)^{\a}\eta^{2} \,dV_{g} \\
&=& - \int_{M^{-}_{s_{2}}} R \Delta R (-R)^{\a}\eta^{2} \,dV_{g} -\a \int_{M^{-}_{s_{2}}} \langle \nabla R, \nabla R \rangle (-R)^{\a}\eta^{2} \,dV_{g}\\
&&+\, 2\int_{M^{-}_{s_{2}}}\langle \nabla R, \nabla \eta \rangle (-R)^{\a+1} \eta \,dV_{g}\,,
\end{eqnarray*}
since the boundary terms vanish. Thus, from equation~\eqref{eq1}, we get
\begin{eqnarray*}
\int_{M^{-}_{s_{2}}}|\nabla R|^{2} |R|^{\a} \eta^{2}\,dV_{g}
&=& - \frac{1}{1+\a}\int_{M^{-}_{s_{2}}} R \Delta R (-R)^{\a}\eta^{2} \,dV_{g} + \frac{2}{1+\a}\int_{M^{-}_{s_{2}}}\langle \nabla R, \nabla \eta \rangle |R|^{\a+1} \eta \,dV_{g} \\
&=& \frac{n-4}{4(n-1)(1+\a)} \int_{M^{-}_{s_{2}}}|R|^{\a+3}\eta^{2}\,dV_{g} + \frac{2}{1+\a}\int_{M^{-}_{s_{2}}}\langle \nabla R, \nabla \eta \rangle |R|^{\a+1} \eta \,dV_{g} \,.
\end{eqnarray*}
On the other hand, Schwartz inequality implies
$$
\frac{2}{1+\a}\int_{M^{-}_{s_{2}}}\langle \nabla R, \nabla \eta \rangle |R|^{\a+1} \eta \,dV_{g} \leq \frac{\eps}{1+\a} \int_{M^{-}_{s_{2}}}|\nabla R|^{2} |R|^{\a} \eta^{2}\,dV_{g} + \frac{1}{\eps(1+\a)}\int_{M^{-}_{s_{2}}} |R|^{\a+2} |\nabla \eta|^{2} \,dV_{g} \,,
$$
for every $\eps>0$. Choosing $\eps=(1+\a)/2$ we get the result.
\end{proof}

\begin{proof}[Proof of Theorem \ref{t-3d} and Remark \ref{r-3d}]
First of all, if $M^3$ is compact, then integrating \eqref{eq1} over $M^3$ we get $R\equiv 0$ on $M^3$.

\smallskip

On the other hand, suppose that $M^3$ is non-compact
and let $1<q<\infty$. Using Lemma~\ref{estimate} with $n=3$, $s_{2}=2s_{1}=2s>0$ and $\alpha=q-2>-1$ we get
$$
\int_{M^{-}_{s}}|\nabla R|^{2}|R|^{q-2} \,dV_{g} \leq \frac{4}{(q-1)^{2}s^{2}} \int_{M^{n}}|R|^{q} \,dV_{g} \longrightarrow 0\quad\text{as}\quad s\to +\infty.
$$
Hence we have that $(M^{3},g)$ has non-negative scalar curvature. It follows from~\cite[Theorem 1.2]{cat1} that $(M^{3},g)$ has to be scalar flat, and Theorem~\ref{t-3d} and Remark \ref{r-3d} in this case are proved.
\end{proof}

\begin{rem}\label{rem-not-And}
We explicitly note that the strategy used in the proof of Theorem \ref{t-3dFt} cannot be applied in the case of critical metrics of $\mathfrak{S}^2$. More precisely,  even if the rigidity result of Lemma \ref{l-rig} still holds (see Remark \ref{remmdn}), the local regularity of $L^{2,2}$-weak solutions given by Lemma \ref{l-regularity}  could be proved (possibly with some effort, due to the presence of the boundary set $\Sigma=\set{R=0}$), and thus also the blow-up argument in Lemma \ref{l-blowup} works, the contradiction in the proof of Theorem \ref{t-3dFt} cannot be achieved, since the solution $(\bar{M}^3, \bar{g})$ must only be scalar flat.
\end{rem}

\

\section{Higher dimensions}\label{sec5}

The aim of this section is to prove Theorem \ref{t-hd} and Remark \ref{r-hd}.

First of all, we show the following

\begin{lemma}\label{l-hdnonpos} Let $(M^{n},g)$, $n>4$, be a complete, non-compact, critical metric of $\mathfrak{S}^2$ with $R\in L^{q}(M^{n})$ for some $1<q<\infty$. Then $(M^{n},g)$ has non-positive scalar curvature.
\end{lemma}
\begin{proof} We  prove this by contradiction. Assume that there exists a point $O\in M^{n}$ such that $R(O)>0$. For $s>0$, we define the open set
$$
M^{+}_{s} = \{ p\in M^{n}|\, R(p)>0 \} \cap B_{s} (O) \,,
$$
where $B_{s}(O)$ is the geodesic ball of radius $s$ centered in $O$. Let $\eta$ be a smooth cutoff function such that $\eta\equiv1$ on $B_s(O)$, $\eta\equiv0$ on $B_{2s}^c(O)$, $0\leq\eta\leq1$ on $M^n$ and $|\nabla\eta|\leq\frac{c}{s}$, $c>0$ independent of $s$. From equation~\eqref{eq1}, an integration by part and an application of Young's inequality yield

\begin{eqnarray*}
\int_{M^{+}_{2s}}R^{q+1} \eta^{2} \,dV_{g} &=& \frac{4(n-1)}{n-4} \int_{M^{+}_{2s}} R^{q-1} \Delta R \,\eta^{2} \,dV_{g} \\
&=& - \frac{4(n-1)(q-1)}{n-4}\int_{M^{+}_{2s}}|\nabla R|^{2} R^{q-2} \eta^{2} \,dV_{g} \\
&& - \frac{8(n-1)}{n-4} \int_{M^{+}_{2s}} \langle \nabla R, \nabla\eta \rangle R^{q-1} \eta \,dV_{g} \\
&\leq& \frac{4(n-1)}{(n-4)(q-1)}\int_{M^{+}_{2s}} R^{q} |\nabla \eta|^{2} \,dV_{g} \\
&\leq& \frac{4(n-1)}{(n-4)(q-1)s^{2}} \int_{M^{n}} |R|^{q} \,dV_{g} \,.
\end{eqnarray*}
By letting $s\rightarrow +\infty$, we get that the set $M^{+}= \{p\in M^{n}| R(p)>0\}$ must have zero measure, so it must be empty: this is a contradiction, since $O\in M^{+}$.

\end{proof}
Now, if $n>4$, then by \eqref{eq1} $R$ is subharmonic, therefore Lemma \ref{l-hdnonpos} and the strong maximum principle imply the following:
\begin{cor}\label{c-strong} Let $(M^{n},g)$, $n>4$, be a complete, non-compact critical metric of $\mathfrak{S}^2$ with $R\in L^{q}(M^{n})$ for some $1<q<\infty$. Then $(M^{n},g)$ is either scalar flat or it has negative scalar curvature.
\end{cor}

From now on we will assume that $(M^{n},g)$, $n>4$, is a complete, non-compact, critical metric of $\mathfrak{S}^2$ with $R\in L^{q}(M^{n})$ for some $1<q<\infty$ and with negative, bounded from below scalar curvature. Let $f=-\log(-R)$; one has
$$
\nabla f = -\frac{\nabla R}{R} \quad\quad\hbox{and}\quad\quad \nabla^{2} f = -\frac{\nabla^{2}R}{R} + df \otimes df \,.
$$
Hence, the structure equation~\eqref{eq} implies that the $1$--Bakry--Emery Ricci tensor with potential function $f$, i.e. $\Ric^{1}_{f} = \Ric + \nabla^{2} f - df \otimes df$ (see e.g. \cite{WW}), satisfies
\begin{equation}\label{eqf}
\Ric^{1}_{f}  = - \frac{3}{4(n-1)} \,e^{-f}\,g \,.
\end{equation}
We explicitly note that the (global) change of variable $f=-\log(-R)$, which is permitted by Corollary \ref{c-strong}, allows to read a critical metric as a special Einstein-type manifold, in the sense of \cite{CMMR} (see also \cite{CMbook}). We now aim at proving gradient estimates for the function $f$, which satisfies the semilinear equation for the $f$-Laplacian
\begin{equation}\label{eqft}
\Delta_{f}f = \Delta f - \abs{\nabla f}^2= \frac{n-4}{4(n-1)} e^{-f} \,
\end{equation}
that one obtains by tracing \eqref{eqf}. It is well known that this kind of results for the standard, ``unweighted'' Laplacian can be obtained only assuming a lower bound on the Ricci tensor, which ensure, in particular, the validity of the Laplacian comparison and of the Omori-Yau maximum principle.
Since $R$ is bounded from below (or, equivalently, $e^{-f}$ is bounded from above), we see that the ``weighted'' $\Ric^{1}_{f}$ tensor is bounded from below: this will allows us to obtain the following Lemmas \ref{gradient estimate} and \ref{lemma5}.

\begin{lemma}\label{gradient estimate}
Let $(M^{n},g)$, $n>4$, be a complete, non-compact, critical metric of $\mathfrak{S}^2$ with negative, bounded from below scalar curvature. Then there exists a positive constant $C$, only depending on $n$ and the lower bound of the scalar curvature, such that, on $M^{n}$, there holds
\begin{equation}\label{12}
  |\nabla R|^{2} \leq C \, R^{2} \,.
\end{equation}
\end{lemma}

\begin{proof} Let $O\in M^{n}$ be some origin point, denote by $B_{s}$ the geodesic ball with radius $s>0$ centered in $O$ and let
$$
Z_{s}^{2}:= \sup_{B_{s}} e^{-f}\,.
$$
We will show that there exist three positive constants $c_{1}, c_{2}, c_{3}$, just depending on $n$, such that
\begin{equation}\label{eqclaim}
|\nabla f|^{2} (p) \,\leq \frac{c_{1}}{s^{2}} + \Big(\frac{c_{2}}{s} + c_{3} Z_s\Big) Z_{s} \,,
\end{equation}
for every $s>0$ and every $p \in B_{s/2}$. The global statement will follow by letting $s\rightarrow \infty$. Using~\eqref{eqf},~\eqref{eqft} and the Bochner formula applied to $f$ we get
$$
\frac{1}{2} \Delta_f |\nabla f|^2 = |\nabla^2f|^2 + g( \nabla \Delta_f f, \nabla f) + \Ric_f^1(\nabla f, \nabla f) +  |\nabla f|^4.
$$
Now we use Newton inequalities and $\Delta f = \Delta_ff + |\nabla f|^2$ to obtain
\begin{align*}
\frac {1}{2} \Delta_f |\nabla f|^2 &= |\nabla^2f|^2 -   \frac{n-4}{4(n-1)} e^{-f}|\nabla f|^2   - \frac{3}{4(n-1)} \,e^{-f}|\nabla f|^2 + |\nabla f|^4 \\
&=|\nabla^2f|^2 -\frac{1}{4} e^{-f}|\nabla f|^2  + |\nabla f|^4 \\
&\geq \frac {1}{n}(\Delta f)^2 -\frac{1}{4} e^{-f}|\nabla f|^2+|\nabla f|^4 \\
&\geq \big(1+\frac {1}{n}\big) |\nabla f|^4 - \frac{1}{4} e^{-f}|\nabla f|^2,
\end{align*}
and then we deduce
\begin{equation}\label{estimDeltaffabs}
\Delta_f |\nabla f|^2 \geq \frac{2n+2}{n} |\nabla f|^4-\frac{1}{2}e^{-f}|\nabla f|^2.
\end{equation}
We note that, on $B_s(O)$,
      \[
      \Ric^{1}_{f} - \frac{3}{4(n-1)} \,e^{f}\,g \geq - \frac{3}{4(n-1)}Z_{s}^{2} g \,.
      \]
Now we proceed exactly as in Theorem 4 on \cite{mas_rim}. We include here the details for the sake of completeness. Let  $\rho(x) := dist(O, x)$: using the Calabi trick (\cite{calabi} and \cite{cat1} for details) we can suppose that $\rho$ is smooth and consider on $B_s(O)$ the function
      \begin{equation}
        F(x) = \left[s^2 - \rho^2(x)\right]^2 |\nabla f|^2.
      \end{equation}
      If $|\nabla f| \equiv 0$ we have nothing to prove; if $|\nabla f| \not \equiv 0$, since $F \geq 0$ and  $\left. F \right|_{\partial B(O, s)} \equiv 0$, there exists a point $x_0 \in B(O,s)$ such that $F(x_0) = \underset{\overline{B_s(O)}}\max F(x) > 0$. At $x_0$ we then have
      \begin{equation}\label{nablaFoverF}
      \nabla F(x_0) = 0,
      \end{equation}
      \begin{equation}\label{deltafFoverF}
       \Delta_f F(x_0) \leq 0.
      \end{equation}
A straightforward calculation shows that \eqref{nablaFoverF} is equivalent to
      \begin{equation}\label{nablaFoverFtranslated}
      \frac{\nabla |\nabla f|^2}{|\nabla f|^2}= \frac{2\nabla \rho^2}{s^2-\rho^2} \qquad \text{ at } x_0,
      \end{equation}
      while, using \eqref{nablaFoverFtranslated} condition \eqref{deltafFoverF} is equivalent to
       \begin{equation}\label{deltafFoverFtranslated}
       0 \geq -2\frac{\Delta_f \rho^2}{s^2-\rho^2} + \frac{\Delta_f|\nabla f|^2}{|\nabla f|^2} -24 \frac{\rho^2}{(s^2-\rho^2)^2} \quad \text{ at } x_0.
      \end{equation}
       From the $f$-Laplacian comparison theorem (see \cite{qian}, \cite{WW}), on $ B_s(O)$ we have
      \begin{equation}\label{fLaplcomparisonZ}
        \Delta_f \rho^2 \leq 2\left[(n+1) + nY\rho\right],
      \end{equation}
      where $Y^2 := \frac{3}{4n(n-1)}Z_{s}^{2}$.
      Combining \eqref{estimDeltaffabs}, \eqref{deltafFoverFtranslated} and \eqref{fLaplcomparisonZ} we find, at $x_0$,
      \begin{equation*}
        0 \geq -4 \frac{\left[(n+1) + nY\rho\right]}{s^2-\rho^2} + \frac{2n+2}{n}|\nabla f|^2 - \frac{1}{2} e^{-f}-24 \frac{\rho^2}{(s^2-\rho^2)^2},
      \end{equation*}
      which implies, multiplying through by $(s^2-\rho^2)^2$, that at $x_0$ we have
      \begin{equation}
        0 \geq -4 \left[(n+1) + nY\rho\right](s^2-\rho^2) + \frac{2n+2}{n} F -\frac{1}{2}(s^2-\rho^2)^{2} e^{-f}-24 \rho^2.
      \end{equation}
From this, claim~\eqref{eqclaim} follows and this concludes the proof of the lemma.
\end{proof}

Now we aim at improving the gradient estimate \eqref{12}, and to obtain that under the same assumptions of Lemma \ref{gradient estimate} one has
\begin{equation}\label{13}
  |\nabla R|^{2} \leq C \, |R|^{3}\quad\text{ on }M^{n}
\end{equation}
for some positive constant $C$, only depending on $n$. We start with the following

\begin{lemma}\label{l-hessest}
Let $(M^n,g)$ be $n$-dimensional Riemannian manifold and let $w\in C^2(M^n)$. Then, where $|\nabla w|\neq 0$, it holds
$$
|\nabla^2 w|^2\geq \frac{1}{n-1}(\Delta w)^2-\frac{1}{n-1}\frac{\Delta w}{|\nabla w|^2}\langle\nabla|\nabla w|^2,\nabla w\rangle.
$$
\end{lemma}
\begin{proof} Let $\{e_{1},\ldots,e_{n}\}$ be a local orthonormal frame with $e_{1}=|\nabla w|^{-1}\nabla w$. Then
\begin{align*}
|\nabla^{2} w|^{2} \geq&\, \sum_{i=2}^{n}(\nabla_{ii} w)^{2}\geq \frac{1}{n-1}\left( \sum_{i=2}^{n}\nabla_{ii} w \right)^{2} = \frac{1}{n-1}\left(\Delta w -\nabla_{11}w\right)^{2} \\
\geq& \frac{1}{n-1} (\Delta w)^{2}-\frac{2}{n-1}\Delta w \nabla_{11} w \,.
\end{align*}
Now, noting that
$$
\nabla_{11} w = \nabla^{2} w (e_{1},e_{1}) = \frac{1}{2} |\nabla w|^{-2} \pair{\nabla |\nabla w|^{2},\nabla w} \,,
$$
we obtain the desired estimate.

\end{proof}

We can now prove the following
\begin{lemma}\label{lemma5}
Let $(M^{n},g)$, $n>4$, be a complete, non-compact, critical metric of $\mathfrak{S}^2$ with negative, bounded from below scalar curvature. There exists a positive constant $C$, only depending on $n$, such that, on $M^{n}$, there holds
\begin{equation}\label{18}
  |\nabla R|^{2} \leq C \, |R|^{3} \,.
\end{equation}
\end{lemma}
\begin{proof}
  Let again $f=-\log(-R)$ and let $u=-R=e^{-f}$. By the Bochner formula, using \eqref{eqf}, \eqref{eqft} and Lemma \ref{l-hessest}, we see that at those points of $M^n$ where $|\nabla f|\neq0$ we have
\begin{align*}
  \frac{1}{2} \Delta_f |\nabla f|^2 &= |\nabla^2f|^2 + \langle \nabla \Delta_f f, \nabla f\rangle + \Ric_f^1(\nabla f, \nabla f) +  |\nabla f|^4\\
   & \geq \frac{1}{n-1}(\Delta f)^2-\frac{1}{n-1}\frac{\Delta f}{|\nabla f|^2}\langle\nabla|\nabla f|^2,\nabla f\rangle-\frac{3}{4(n-1)}u|\nabla f|^2\\
   &\,\,\,\,\,\,\,+|\nabla f|^4-\frac{n-4}{4(n-1)}u|\nabla f|^2\\
   &=\frac{1}{n-1}\left(\frac{n-4}{4(n-1)}u+|\nabla f|^2\right)^2-\frac{1}{n-1}\frac{\Delta f}{|\nabla f|^2}\langle\nabla|\nabla f|^2,\nabla f\rangle\\
   &\,\,\,\,\,\,\,+|\nabla f|^4-\frac{1}{4}u|\nabla f|^2\\
   &=-\frac{1}{n-1}\frac{\Delta f}{|\nabla f|^2}\langle\nabla|\nabla f|^2,\nabla f\rangle+\frac{(n-4)^2}{16(n-1)^3}u^2+\frac{n}{n-1}|\nabla f|^4\\
   &\,\,\,\,\,\,\,-\frac{n^2-4n+9}{4(n-1)^2}u|\nabla f|^2\\
\end{align*}
Moreover
$$ \frac{1}{2} \Delta_f u=-\frac{n-4}{8(n-1)}u^2+\frac{1}{2}u|\nabla f|^2.$$
Now let
$$v:=Au-|\nabla f|^2,$$ where $A>0$ is a suitable constant to be chosen later. We claim that we can find $A>0$ such that $v\geq0$ on $M$.

By the above calculations we have
\begin{align}
 \label{14} \frac{1}{2} \Delta_f v &\leq -\frac{n-4}{8(n-1)}Au^2+\frac{A}{2}u|\nabla f|^2+\frac{1}{n-1}\frac{\Delta f}{|\nabla f|^2}\langle\nabla|\nabla f|^2,\nabla f\rangle\\
 \nonumber &\,\,\,\,\,\,\,-\frac{(n-4)^2}{16(n-1)^3}u^2-\frac{n}{n-1}|\nabla f|^4+\frac{n^2-4n+9}{4(n-1)^2}u|\nabla f|^2
\end{align}
By Lemma \ref{gradient estimate} we have that $v$ is bounded below. Since $u$ is bounded, by \eqref{eqf} we have that the $1$--Bakry--Emery Ricci curvature, $\Ric^{1}_{f} $, is bounded below and hence the full Omori--Yau maximum principle holds for the operator $\Delta_f$, see e.g. \cite{mas_rim} and references therein. Thus, there exists a sequence of points $\{x_k\}_k\subset M^n$ such that
\begin{equation*}
v(x_k)\leq\inf v+\frac{1}{k},\qquad |\nabla v(x_k)|\leq\frac{1}{k},\qquad\Delta_fv(x_k)\geq-\frac{1}{k},\qquad \text{ for all }k\in\mathbb{N}.
\end{equation*}

Since by Lemma \ref{gradient estimate} we have that $|\nabla f|$ is bounded, up to extracting a subsequence $\{y_m\}_m\subset\{x_k\}_k$ we can assume that $$\lim_{m\rightarrow+\infty}|\nabla f(y_m)|=l,$$
with
\begin{equation}\label{17}
v(y_m)\leq\inf v+o(1),\qquad |\nabla v(y_m)|\leq o(1),\qquad\Delta_fv(y_m)\geq-o(1),\qquad \text{ and } o(1)>0
\end{equation}
as $m$ tends to $\infty$.

Then one has
$$Au(y_m)-|\nabla f(y_m)|^2=v(y_m)\leq\inf v+o(1)$$
as $m$ tends to $+\infty$. Since $u$ is positive on $M^n$, if $l=0$, passing to the limit as $m$ tends to $\infty$ in the above equation yields $\inf v\geq0$ (for any choice of $A>0$).

From now on we assume by contradiction that $l>0$, and without loss of generality that $|\nabla f(y_m)|>0$ for every $m\in\mathbb{N}$. Then, noting that $$\nabla|\nabla f|^2=A\nabla u - \nabla v$$ and using \eqref{eqft}, at $y_m$ we have
\begin{align*}
\frac{1}{n-1}\frac{\Delta f}{|\nabla f|^2}\langle\nabla|\nabla f|^2,\nabla f\rangle&=\frac{1}{n-1}\frac{\Delta f}{|\nabla f|^2}\left(A\langle\nabla u,\nabla f\rangle-\langle\nabla v,\nabla f\rangle\right)\\
&=\frac{1}{n-1}\,\frac{1}{|\nabla f|^2}\left(\frac{n-4}{4(n-1)}u+|\nabla f|^2\right)\left(-Au|\nabla f|^2-\langle\nabla v,\nabla f\rangle\right)\\
&\leq-\frac{A}{n-1}\left(\frac{n-4}{4(n-1)}u^2+u|\nabla f|^2\right)+o(1)
\end{align*}
as $m$ tends to $\infty$. Here we have used that $u$ is bounded, that $|\nabla f(y_m)|$ converges to $l>0$ and that $|\nabla v(y_m)|\leq o(1)$ as $m$ tends to $\infty$. Inserting the above inequality into \eqref{14}, we obtain that at $y_m$
\begin{align}
\label{16}\frac{1}{2} \Delta_f v &\leq -\frac{n-4}{16(n-1)^3}(2A(n^2-1)+n-4)u^2-\frac{n}{n-1}|\nabla f|^4\\
\nonumber  &\,\,\,\,\,\,\,+\frac{1}{4(n-1)^2}(2A(n-1)(n-3)+n^2-4n+9)u|\nabla f|^2+o(1)\\
\nonumber  &\,\,\,\,\,\,\,=\left(-\alpha\left(\frac{u}{|\nabla f|^2}\right)^2+\beta\frac{u}{|\nabla f|^2}-\gamma\right)|\nabla f|^4+o(1),
\end{align}
with $$\alpha=\frac{(n-4)(2A(n^2-1)+n-4)}{16(n-1)^3},\,\,\,\,\beta=\frac{(2A(n-1)(n-3)+n^2-4n+9)}{4(n-1)^2},\,\,\,\,\gamma=\frac{n}{n-1},$$
$\alpha,\beta,\gamma>0$ for $n>4$. Now note that
$$\Delta_1=\beta^2-4\alpha\gamma=\frac{1}{16(n-1)^3}(aA^2-bA+c),$$
with $$a=4(n-1)(n-3)^2,\,\,\,\,b=4(n^3+n^2-29n+27),\,\,\,\,c=n^3-11n^2+55n-81,$$
$a,b,c>0$ for $n>4$. Since for every $n>4$
$$\Delta_2=b^2-4ac=64n(n-1)(n-4)(n(5n-26)+9)>0$$
we can choose $A>0$ such that $\Delta_1<0$. In particular, any $A$ satisfying
\begin{equation}\label{15}
A\in \left(\frac{b-\sqrt{\Delta_2}}{2a},\frac{b+\sqrt{\Delta_2}}{2a}\right)
\end{equation}
will do, since $\frac{b-\sqrt{\Delta_2}}{2a}>0$. For every $A$ satisfying \eqref{15} there exist a constant $\lambda>0$ such that
$$-\alpha t^2+\beta t-\gamma\leq-\lambda<0\qquad\text{for every }t\in\mathbb{R}.$$
Hence from \eqref{17} and \eqref{16} we deduce that
$$-o(1)\leq \frac{1}{2} \Delta_f v (y_m)\leq-\lambda |\nabla f(y_m)|^4+o(1),$$
which leads to the contradiction $0\leq-\lambda l^4$ as $m$ tends to $\infty$, since we assumed $l>0$.

Thus for every $A$ satisfying \eqref{15} one has $l=0$ and hence $v\geq0$ on $M^n$, which implies \eqref{18} with $C=A$.
\end{proof}

\begin{rem}\label{remC}
  As it is clear from the proof of Lemma \ref{lemma5}, one has the gradient estimate \eqref{18} where $C$ can be chosen to be the infimum of the $A's$ for which the proof goes through, i.e.
  \begin{equation*}
  C=\frac{b-\sqrt{\Delta_2}}{2a}=\frac{4(n^3+n^2-29n+27)-\sqrt{64n(n-1)(n-4)(n(5n-26)+9)}}{8(n-1)(n-3)^2}
  \end{equation*}
\end{rem}

\

As a corollary of the gradient estimate \eqref{18}, we obtain the following decay estimate at infinity for $R$, which is the last technical result that we will need in the proof of Theorem \ref{t-hd}.

\begin{cor}\label{c-lower} Let $(M^{n}, g)$, $n>4$, be a complete, non-compact, Riemannian manifold satisfying~\eqref{eq} and~\eqref{eq1} with $R<0$ on $M^n$. Let $O\in M^n$. For every $x\in M^n$ we have
\begin{equation}\label{20}
R(x)\leq -\frac{c_1}{c_2+d_{g}(x,O)^2}
\end{equation}
for some positive constants $c_i=c_i(n,u(O))$, i=1,2.
\end{cor}
\begin{proof} Let $u=-R$. Inequality \eqref{18} is equivalent to
$$
|\nabla u^{-1/2}|\leq c
$$
for some positive constant $c=c(n)$. Integrating along a geodesic, we get
$$
u(x)^{-1/2}\leq u(O)^{-1/2}+c\, d_{g}(x,O),
$$
i.e.
$$
\frac{1}{u(x)}\leq \left(\frac{1}{u(O)^{1/2}}+c\, d_{g}(x,O)\right)^2,
$$
from which we immediately deduce \eqref{20}.
\end{proof}

\begin{proof}[Proof of Theorem \ref{t-hd} and Remark \ref{r-hd}] First of all, if $M^n$ is compact, then integrating \eqref{eq1} over $M^n$ we get $R\equiv 0$ on $M^n$.

\smallskip

On the other hand, suppose that $M^n$ is non-compact. By Corollary \ref{c-strong} either $(M^n,g)$ is scalar flat, or $R<0$ on $M^n$. We assume by contradiction that $R<0$ and we recall that $u=-R$ satisfies
$$
\Delta u = - \frac{n-4}{4(n-1)}u^2.
$$
Let $\eta$ be a smooth cutoff function such that $\eta\equiv1$ on $B_s(O)$, $\eta\equiv0$ on $B_{2s}^c(O)$, $0\leq\eta\leq1$ on $M^n$ and $|\nabla\eta|\leq\frac{c}{s}$ for every $s\gg1$ with $c>0$ independent of $s$.
Then, using \eqref{18} we get
\begin{align*}
\frac{n-4}{4(n-1)}\int_M u^q \eta^2\,dV_{g} &= -\int_M \Delta u\, u^{q-2}\eta^2\,dV_{g} \\
&= (q-2) \int_M |\nabla u|^2 u^{q-3}\eta^2\,dV_{g} +2\int_M u^{q-2}\langle \nabla u,\nabla \eta\rangle \eta\,dV_{g}\\
&\leq (q-2) C\int_M u^{q}\eta^2\,dV_{g} +\frac{c\sqrt{C}}{s}\int_{B_{2s}(O)\setminus B_s(O)}u^{q-\frac{1}{2}}\,dV_{g},
\end{align*}
with $C$ as in \eqref{18}. By Corollary \ref{c-lower}
\begin{equation}\label{11}
\frac{n-4}{4(n-1)}\int_M u^q \eta^2\,dV_{g}\leq(q-2)C \int_M u^{q}\eta^2\,dV_{g} +c\sqrt{C}\frac{(1+s^2)^\frac{1}{2}}{s}\int_{B_s^c(O)}u^{q}\,dV_{g}.
\end{equation}
Thus, if $u\in L^q(M^n)$, we obtain
\begin{equation*}
\left(\frac{n-4}{4(n-1)}-(q-2)C\right)\int_M u^q \eta^2\,dV_{g}\leq c\sqrt{C}\frac{(1+s^2)^\frac{1}{2}}{s}\int_{B_s^c(O)}u^{q}\,dV_{g}\longrightarrow 0\quad\text{as}\quad s\to +\infty.
\end{equation*}
This yields $u\equiv0$, if $$\frac{n-4}{4(n-1)}-(q-2)C>0,$$ i.e. if $$1<q<q^*=\frac{n-4}{4C(n-1)}+2$$ with $C$ as in \eqref{18}, a contradiction. Hence, the proofs of Theorem~\ref{t-hd} and Remark \ref{r-hd} are complete.

\end{proof}

\begin{rem}\label{remf}
Note that $q^*>2$, and that using Remark \ref{remC} we have
$$q^*=2+\frac{2(n-3)^2(n-4)}{4(n^3+n^2-29n+27)-\sqrt{64n(n-1)(n-4)(n(5n-26)+9)}}.$$
\end{rem}

\

\

\

\bibliographystyle{amsplain}
\bibliography{criticalFin}

\

\

\

\parindent=0pt

\end{document}